\documentclass[final]{siamltex}
\usepackage[english]{babel}
\usepackage{graphics}
\usepackage{graphicx}
\usepackage{amssymb}
\usepackage{arydshln}
\usepackage{amsmath}
\usepackage{amsfonts}
\usepackage{float}
\usepackage{multirow} 
\usepackage{booktabs} 
\usepackage{pstricks,pst-plot}
\usepackage{hyperref}
\usepackage{epstopdf}
\usepackage{subfigure}
\usepackage{dsfont} 
\usepackage{caption}
\usepackage{enumerate}
    \newtheorem{remark}[theorem]{Remark}

\newcommand{\be}{\begin{equation}}
\newcommand{\ee}{\end{equation}}
\newcommand{\ba}{\begin{array}}
\newcommand{\ea}{\end{array}}

\newcommand{\comment}[1]{}

 \newcommand{\B}[1]{\mbox{\boldmath $#1$}}

\newenvironment{code1}{%
                           \mathcode`\:="603A  % TeXbook pp 134, 154, 359 (top)
                           \def\colon{\mathchar"303A}
                           \par
                           \upshape
                           \begin{list} % To give indentation
                              {} {\leftmargin = 0.0cm}
                           \item[]
                           \begin{tabbing}
                              \hspace*{.3in} \= \hspace*{.3in} \=
                              \hspace*{.3in} \= \hspace*{.3in} \=
                                                                                                                                                                                                     \hspace*{.3in} \= \hspace*{.3in} \= \kill
                          }{\end{tabbing}\end{list}}

\title{A CMV--based eigensolver for companion matrices}

\author{R. Bevilacqua\footnotemark[1]
\and G. M. Del Corso\footnotemark[1]
\and L. Gemignani\footnotemark[1]~\footnotemark[5]}

\begin{document}

\maketitle

\renewcommand{\thefootnote}{\fnsymbol{footnote}}

\footnotetext[1]{Dipartimento di Informatica, Universit\`a di Pisa,
Largo Bruno Pontecorvo 3, 56127 Pisa, Italy, \{bevilacq,delcorso,l.gemignani\}@di.unipi.it}
\footnotetext[5]{This work was partially supported by GNCS-INDAM.}

\begin{abstract}

In this paper we present a novel matrix method for polynomial rootfinding.   By exploiting the 
 properties  of the QR eigenvalue algorithm  applied to a suitable 
CMV--like  form of a companion matrix  we design a fast and computationally simple 
structured QR iteration. 
  
\end{abstract}


\section{Introduction}
\setcounter{equation}{0}

This paper  stems from two research lines  which blend in the effective solution  of certain 
 eigenproblems for companion--like matrices arising in polynomial rootfinding.  The first one 
begins with the exploitation of the structure of companion-like matrices under the 
QR eigenvalue algorithm. In the recent years based on the concept of rank structure 
many authors have provided fast 
adaptations of the QR iteration applied to small rank modifications of Hermitian  or 
unitary matrices. However, despite the common framework, there are several 
significant differences 
between the Hermitian and the unitary case which makes the latter  much more involved computationally.
The second line originates from the treatment of the unitary eigenproblem. It has been observed 
in the seminal
 paper 
\cite{BE91} that the 
CMV-like banded form of a unitary matrix rather than its Hessenberg reduction leads to a QR--type 
algorithm which is 
ideally close to the Hermitian tridiagonal QR algorithm as it maintains the band shape  of the initial 
matrix at any step.  The  present work  lies at the intersection of these two strands and 
is specifically aimed to incorporate  
the CMV technology for the unitary eigenproblem 
in the design of   fast QR--based eigensolvers  for companion--like matrices. 

The first fast  structured variant of the QR iteration for companion matrices was proposed  
in \cite{BDG}. The invariance of the rank properties of the matrices generated by the QR scheme  is 
captured by means of three rank--one matrices which are  easily updated under the iterative process.
Since the representation breaks down for reducible Hessenberg matrices  the price paid to keep 
the algorithm simple 
is a progressive deterioration  in the limit of the accuracy of computed eigenvalues. 
Overcoming this drawback  is the 
main subject of many subsequents papers \cite{BBEGG,BEGG_simax,BEGG_comp,CG,VVVF}, where 
more refined  parametrizations of the rank structure are employed. While this leads to
numerically stable methods, it also opens the way to involved algorithms which 
exhibit worse timing performance and are difficult to generalize  to the block  matrix/pencil case.
This is astonishingly unpleasant when compared with the  simplicity and the effectiveness 
of adjusting the QR  scheme 
for perturbed Hermitian matrices \cite{EGGnew,VDC}.

The approach pursued here moves  away from the classical scenario where nonsymmetric 
matrices are 
converted in Hessenberg form for eigenvalue computation, 
focusing instead  on a preliminary reduction of a 
companion matrix $A\in \mathbb C^{n\times n}$  into a  different staircase form.  More specifically,  
recall that 
 $A\in \mathbb C^{n\times n}$ 
can be  expressed as a rank--one correction of a  unitary matrix $U$ 
generating the circulant matrix algebra.  The transformation of $U$ by unitary congruence 
into a CMV--like form \cite{CMV,KN}  induces a corresponding reduction of the matrix $A$ into an  
upper block Hessenberg form with a certain specified staircase pattern. The CMV-like form of a unitary matrix is 
particularly suited for the application of the QR eigenvalue algorithm \cite{BE91}. 
The staircase shape also reveals invariance properties under the same algorithm \cite{AG}. 

From 
these properties it follows that the matrices generated by the shifted QR method applied to the transformed 
companion matrix  inherit a simplified rank structure which can be expressed in terms of two rank--one matrices. 
This yields a data sparse parametrization of each matrix
which 
at the same time is able to capture the
structural properties of the matrix and yet  to be very  easy to manipulate and update
 for computations.  We shall develop a fast adaptation of the 
QR eigenvalue algorithm for companion matrices that exploits this parametrization and requires 
$O(n)$ arithmetic operations per step.  The main complexity  of the algorithm
lies in updating the narrow  diagonal staircase  of each matrix. The  results from 
numerical experiments indicate that the proposed approach works stable and efficient.

The paper is organized as follows. 
In Section 2,  we first recall some preliminaries about CMV--like representations  of unitary matrices and then 
introduce  the considered reduction of a companion matrix. The structural properties of the modified matrix 
under the shifted QR iteration are analyzed in Section 3.  In Section 4 we  present our  fast 
adaptation of the shifted 
QR algorithm for companion matrices and report the results of numerical experiments. 
Finally, in Section 5
 the conclusion and  further developments are drawn.

\section{Preliminaries}
\setcounter{equation}{0}

For  a given  pair $(\gamma, k)\in \mathbb D\times \mathbb I_n$, $\mathbb D=\{z\in \mathbb C\colon |z|< 1\}$, 
$\mathbb I_n=\{1,2,\ldots,n-1\}$,  we set 
\[
\mathcal G_k(\gamma)=I_{k-1}\oplus 
\left[\begin{array}{cc} \bar\gamma & \sigma \\
\sigma & -\gamma \end{array}\right] \oplus I_{n-k-1} \in \mathbb C^{n\times n}, 
\]
where $\sigma\in \mathbb R, \sigma> 0$ and $|\gamma|^2+\sigma^2=1$.  Similarly, if $\gamma \in \mathbb S^{1}= 
\{z\in \mathbb C\colon |z|= 1\}$ then denote
\[
\mathcal G_n(\gamma)=I_{n-1}\oplus 
\gamma \in \mathbb C^{n\times n}.
\]
Observe that $\mathcal G_k(\gamma)$, $1\leq k\leq n$,  is a unitary matrix. 
Given coefficients $\gamma_1, \ldots, \gamma_{n-1}\in \mathbb D$ and $\gamma_n\in \mathbb S^{1}$ we introduce the 
unitary block diagonal matrices 
\[
\mathcal L=\mathcal G_1(\gamma_1)\cdot\mathcal G_3(\gamma_3) \cdots \mathcal 
G_{2 \lfloor \frac{n+1}{2}\rfloor-1}(\gamma_{2 \lfloor \frac{n+1}{2}\rfloor-1}), 
\quad \mathcal M=\mathcal G_2(\gamma_2)\cdot\mathcal G_4(\gamma_4) \cdots \mathcal 
G_{2 \lfloor \frac{n}{2}\rfloor}(\gamma_{2 \lfloor \frac{n}{2}\rfloor}), 
\]
and define 
\begin{equation}\label{Schur}
\mathcal C=\mathcal L\cdot \mathcal M
\end{equation}
 as the CMV matrix associated with the 
prescribed coefficient list \cite{CMV}. 
The decomposition  \eqref{Schur} of a unitary matrix  was first 
investigated for eigenvalue computation in  \cite{BE91}. The staircase 
shape of CMV matrices is analyzed in \cite{KN} where the next  definition is given. 
A matrix $A\in \mathbb C^{n\times n}$  has CMV shape if  the possibly 
nonzero entries  exhibit  the following pattern where $+$ denotes a positive entry:
\[
A=\left[\begin{array}{cccccccc}
\star & \star & +\\
+ & \star & \star \\
& \star & \star & \star & +\\
& + & \star & \star & \star\\
& & &  \star & \star & \star & +\\
&  & & + & \star & \star & \star\\
& & &  & &  \star & \star & \star \\
&  & &  & & + & \star & \star \\
\end{array}\right], \quad \quad  (n=2k),
\]
or 
\[
A=\left[\begin{array}{ccccccc}
\star & \star & +\\
+ & \star & \star \\
& \star & \star & \star & +\\
& + & \star & \star & \star\\
& & &  \star & \star & \star & +\\
&  & & + & \star & \star & \star\\
& & &  &  & \star & \star \\
\end{array}\right], \quad  \quad (n=2k-1).
\]
Obviously, CMV matrices have a CMV shape and, conversely, a unitary matrix 
with CMV shape is CMV \cite{CMV1}.  By skipping the positivity condition 
 in \cite{BDG1} the fairly more general 
class of CMV--like shaped 
matrices is considered.  There  it is shown that the block Lanczos method can be used to 
reduce a unitary matrix into the direct sum of CMV--like shaped 
matrices.

Staircase matrix patterns can be exploited for eigenvalue computation \cite{AG}.  The shifted QR algorithm 
\begin{equation}\label{qrgeneral}
\left\{ \begin{array}{l l}
   A_s-\rho_s I_n=Q_sR_s \\
   A_{s+1}=Q_s^H  A_s Q_s,\quad s\geq 0,
   \end{array}  \right.
\end{equation}
is the standard algorithm for computing the Schur form of a general matrix $A=A_0\in \mathbb C^{n\times n}$ \cite{MC}.
The matrix $A$ is said to be staircase if $m_j(A)\geq m_{j-1}(A)$, $2\leq j\leq n$, where 
\[
m_j(A)=\max\{j, \max_{i>j}\{i\colon a_{i,j}\neq 0\}\}.
\]
The staircase form is  preserved under the QR iteration \eqref{qrgeneral} in the sense that \cite{AG}
\[
m_j(A_{s+1})\leq m_j(A_{s}), \quad 1\leq j\leq n.
\]
For Hermitian and unitary matrices the staircase form also implies a zero pattern or a rank structure in the upper triangular 
part.  The invariance  of this pattern by the QR algorithm 
 is proved in \cite{AG} for Hermitian matrices and in \cite {BE91} for   unitary CMV--shaped 
matrices. An alternative proof for the unitary case that is suitable for generalizations is given 
in \cite{BDG1} by relying upon the classical nullity theorem \cite{FM}.

\begin{theorem}\label{nullycons}
Suppose $A\in \mathbb C^{n\times n}$ is a nonsingular matrix and $\B \alpha$ and $\B \beta$  to be 
nonempty proper subsets of $\mathbb I_{n+1}\colon =\{1,\ldots, n\}$.  Then 
\[
\rank(A^{-1}(\B \alpha; \B \beta))=\rank(A(\mathbb I_{n+1}\setminus \B \beta; \mathbb I_{n+1}\setminus \B \alpha))
 + |\B \alpha| +|\B \beta| -n, 
\]
where, as usual, $|J|$ denotes the cardinality of the set $J$. 
\end{theorem}

The design of efficient numerical methods for eigenvalue computation of almost Hermitian and unitary 
matrices has recently attracted a lot of attention (see \cite{Yuli_book,GU,Raf_book} and the references given therein).
A motivating  application is given by matrix methods 
for polynomial rootfinding.  From a given $n-$th degree polynomial 
\[
p(z)=p_0+p_1z +\ldots + p_nz^n, \quad (p_n\neq 0),
\]
we can set up the associated companion matrix $C\in \mathbb C^{n\times n}$ in upper Hessenberg form, 
\[
C=C(p)=\left[\begin{array}{cccc}
-\displaystyle\frac{p_{n-1}}{p_n} & -\displaystyle\frac{p_{n-2}}{p_n}&\ldots &-\displaystyle\frac{p_{0}}{p_n}\\
1  & 0 & \ldots & 0\\
&  \ddots & \ddots & \vdots\\
&  & 1 & 0
\end{array}\right].
\]
As $p_n \det(z I -C)= p(z)$ is satisfied,  thus we can  obtain approximations of  the zeros of $p(z)$
by applying a standard eigenvalue  method to the associated companion matrix $C$. This is exactly 
the approach taken by the MATLAB function {\tt roots}. 

In the recent years many fast adaptations of the QR iteration \eqref{qrgeneral} applied to an initial 
  companion matrix $A_0=C$ 
have been proposed \cite{BBEGG,BEGG_simax,BEGG_comp,CG,VVVF}, based on  the decomposition of $C$ as a rank--one 
correction of a  unitary matrix, that is, 
\[
C=U-\B e_1\B p^H= 
\left[\begin{array}{cccc}
0 & \ldots&0 &1\\
1  & 0 & \ldots & 0\\
&  \ddots & \ddots & \vdots\\
&  & 1 & 0
\end{array}\right] - \left[\begin{array}{c}
1\\0\\\vdots\\0
\end{array}\right]\left[\displaystyle\frac{p_{n-1}}{p_n}, 
\displaystyle\frac{p_{n-2}}{p_n},\ldots,\displaystyle\frac{p_{0}}{p_n}+1\right].
\]
In this paper we further elaborate on this decomposition  by developing a different  structured representation. 
Let  $P \in \mathbb  R^{n\times n}$,  $P=(\delta_{i, \pi(j)})$  be the permutation matrix associated with the
permutation given by
\[
\pi \ : \
\mathbb I_{n+1} \rightarrow \mathbb I_{n+1},\quad \pi(1)=1; \ 
\pi(j)=\left\{%
\begin{array}{ll}
    k+1, & \quad \hbox{\ \ if $j=2k$;} \\
    \\
    n-k+1, & \quad \hbox{\ \ if $j=2k+1$. } \\
\end{array}%
\right. 
\]
Then it can be easily verified that the matrix $\widehat U=P^T \cdot U\cdot P$ is a CMV--like shaped 
matrix.  Indeed, we have that the nonzero entries of $\widehat U$ are  precisely $(2,1)$, $(n-1,n)$ and 
those of the form $(2j-1,2j+1)$ and $(2j, 2j+2)$ for $j\geq 1$.  For instance in  the case $n=8$ the nonzero pattern 
looks as follows:
\[
\widehat U=\left[\begin{array}{cccccccc}
 &  & 1\\
1 &  &  \\
&  &  &  & 1\\
& 1 &  &  & \\
& & &   &  &  & 1\\
&  & & 1 &  &  & \\
& & &  & &   &  &1  \\
&  & &  & & 1 &  &  \\
\end{array}\right].
\]
Moreover, since $P^T\B e_1=\B e_1$  it follows that, denoting by $\B{\widehat{p}}=P^T{\B p}$ we have 
\begin{equation}\label{ones}
\widehat C=P^T \cdot C\cdot P=\widehat U -\B e_1\,\B{\widehat{p}}^H,
\end{equation}
is a rank--one correction of a  unitary CMV--like shaped  
matrix in staircase form.  In the next section we investigate the properties of the shifted QR  iteration 
\eqref{qrgeneral} applied to $A_0=\widehat C$ for the computation of the zeros of $p(z)$.

\section{Structural Properties under the QR iteration}
\setcounter{equation}{0}

To put our derivation
 on a firm theoretical ground,  in this section we perform a thorough analysis of the structural properties of 
$A_0=\widehat C$ which are maintained under the shifted QR iteration \eqref{qrgeneral}. 

\begin{remark}\label{rem1}
Several different properties can easily be 
checked by assuming that the matrix $A_s-\sigma_s I_n$  in \eqref{qrgeneral} and, hence, a fortiori $R_s$ 
is invertible.  Clearly,  this might  not   always be  the case but, however,  it is well known that the one--parameter 
matrix function $A_s-\lambda I_n$ is analytic in $\lambda$ and  an analytic QR 
decomposition $A_s-\lambda I_n=Q_s(\lambda) R_s(\lambda)$ of this analytic matrix function exists \cite{DiEi}.
For any given fixed initial pair $(Q_s(\sigma_s),  R_s(\sigma_s))$ we can find a branch of the 
analytic QR decomposition of  $A_s-\lambda I_n$  that passes through $(Q_s(\sigma_s),  R_s(\sigma_s))$. 
Following this path  it makes it possible to extend the proof of the properties that are 
closed in the limit. This is for instance the case of the rank properties.
\end{remark}

It has already been  noticed above that the staircase form of $A_0=C$ is preserved under the shifted QR iteration 
\eqref{qrgeneral}. This means that each  unitary matrix $Q_s$ is also in staircase form.  In particular, 
if
 \[
\mathcal G_k(\gamma, \sigma)=I_{k-1}\oplus 
\left[\begin{array}{cc} \bar\gamma & \sigma \\
\bar \sigma& -\gamma \end{array}\right] \oplus I_{n-k-1} \in \mathbb C^{n\times n}, \quad 1\leq k\leq n-1,
\]
where $\gamma, \sigma\in \mathbb D \cup \mathbb S^{1}$ and $|\gamma|^2+|\sigma|^2=1$, 
denote  generalized Givens  reflectors  then the matrix $Q=Q_s$ can be expressed as 
\begin{equation}\label{qrep}
Q=\mathcal G_1(\bar \gamma_1, \bar \sigma_1)\cdot \mathcal G_{2,3} \cdot  \mathcal G_{4,5} \cdots \mathcal 
G_{2(\lfloor \frac{n+1}{2}\rfloor-2),2(\lfloor \frac{n+1}{2}\rfloor-2)+1 }\cdot  \mathcal G_{n-1}, 
\end{equation}
where 
\[
\mathcal G_{\ell, \ell+1} =\mathcal G_{\ell+1}(\bar \gamma_{\ell+1,1}, \bar \sigma_{\ell+1,1}) \cdot 
\mathcal G_\ell(\bar \gamma_{\ell}, \bar \sigma_{\ell}) \cdot 
\mathcal G_{\ell+1}(\bar \gamma_{\ell+1,2}, \bar \sigma_{\ell+1,2}), 
\]
and 
\[
 \mathcal G_{n-1}=\mathcal G_{n-1}(\bar \gamma_{n-1,1}, \bar \sigma_{n-1,1}) \cdot 
\mathcal G_{n-2}(\bar \gamma_{n-2}, \bar \sigma_{n-2}) \cdot 
\mathcal G_{n-1}(\bar \gamma_{n-1,2}, \bar \sigma_{n-1,2})
\]
 for an even $n$ and 
\[
 \mathcal G_{n-1}=\mathcal G_{n-1}(\bar \gamma_{n-1,1}, \bar \sigma_{n-1,1})
\]
if, otherwise, $n$ is odd.

Since from \eqref{ones} 
\[
A_0=\widehat U-\B e_1\,\B{\widehat{p}}^H=U_0 -\B z_0\B w_0^H
\]
we find that 
\begin{equation}\label{base0}
A_{s+1}=Q_s^H A_s Q_s=Q_s^H(U_s -\B z_s\B w_s^H)Q_s=U_{s+1}-\B z_{s+1}\B w_{s+1}^H, \quad s\geq 0, 
\end{equation}
where 
\begin{equation}\label{base1}
U_{s+1}\colon=Q_s^HU_s Q_s, \quad  \B z_{s+1}\colon =Q_s^H\B z_s, \quad  \B w_{s+1}\colon=Q_s^H\B w_s.
\end{equation}
Theorem~\eqref{1theo} describes the structure of the unitary matrix $U_s$, for any $s\ge 0$. We need the following result characterizing the structure of the $Q$ factor appearing in the $QR$ factorization of $A_s$, $s\ge 0$.

\begin{lemma}\label{0shift}
The unitary factor $Q$ generated by means of a QR factorization of 
$A_s$, $s\geq 0$,  has both a lower and upper staircase profile. Specifically, it holds 
\[
Q(1:2j, 2(j+1)+1:n)=0, \quad  1\leq j\leq \lfloor \frac{n+1}{2}\rfloor-2.
\]
\end{lemma}
\begin{proof}
It has already been observed that since the staircase form of $A_0$ is preserved under the shifted $QR$ iteration~\eqref{qrgeneral}, the unitary factor $Q$, corresponding to the unitary matrix involved in a $QR$ iteration without shift has a lower staircase profile. To prove that $Q$ has also an upper staircase profile,
observe that the matrix $A_0$ is such that 
$\rank(A_0(2j+1:2(j+1), 2j:2j+1))=1$, $1\leq j\leq \lfloor \frac{n+1}{2}\rfloor-1$.  From the argument stated in 
Remark \ref{rem1}  it follows that this 
 rank constraint  is preserved 
under the QR iteration and, specifically, we have 
$\rank(A_s(2j+1:2(j+1), 2j:2j+1))=1$, $1\leq j\leq \lfloor \frac{n+1}{2}\rfloor-1$, for any $s\geq 0$. 
The  same property is also inherited from the unitary factor $Q=Q_s$ generated by means of the 
QR factorization of $A_s$, i.e., $A_s=Q R$.  From Theorem \ref{nullycons} we obtain that 
\[
\begin{array}{ll}
\rank(Q(1:2j, 2(j+1)+1:n))=\rank(Q^H( 2(j+1)+1:n,1:2j))=
\\
\rank(Q(2j+1:n, 1:2(j+1))+(n-2)-n=\rank(Q(2j+1:n, 1:2(j+1))-2.
\end{array}
\]
Hence, by combining the constraint  $\rank(Q(2j+1:2(j+1), 2j:2j+1))=1$ with the 
staircase shape of $Q$ one deduces  that $\rank(Q(2j+1:n, 1:2(j+1))=2$ which implies 
\[
\rank(Q(1:2j, 2(j+1)+1:n))=0, \quad 1\leq j\leq \lfloor \frac{n+1}{2}\rfloor-2.
\]
Equivalently, the relation says that $Q(1:2j, 2(j+1)+1:n)$ is a zero matrix and this concludes the proof.
\end{proof}

Lemma \ref{0shift} can be used to exploit the rank properties of the unitary  matrices $U_s$, $s\geq 0$.
\begin{theorem}\label{1theo}
We have 
\[
\rank(U_s(1:2j, 2(j+1)+1:n))\leq 1, \quad 1\leq j\leq \lfloor \frac{n+1}{2}\rfloor-2, \ s\geq 0.
\]
Moreover, if $A_0$ is invertible then 
\[
U_s(1:2j, 2(j+1)+1:n)=B_s(1:2j, 2(j+1)+1:n),  \quad 1\leq j\leq \lfloor \frac{n+1}{2}\rfloor-2, \ s\geq 0, 
\]
where  
\begin{equation}\label{base2}
B_s=\frac{U_s\B w_s \B z_s^HU_s }{\B z_s^H U_s^H \B w_s-1}=Q_s^H B_{s-1} Q_s, \quad s\geq 1,
\end{equation}
is a rank one matrix. 
\end{theorem}
\begin{proof}
Let $A_s=QR$ be a QR factorization of the matrix $A_s$ assumed invertible.  From 
\[
Q^H A = Q^H (U_s  -\B z_s\B w_s^H) =Q^H U_s- Q^H\B z_s\B w_s^H=R 
\]
we obtain  that 
\[
(Q^H A)^{-H}=Q^H (U_s  -\B z_s\B w_s^H)^{-H}=R^{-H}.
\]
Using the Sherman--Morrison formula \cite{MC}   yields 
\[
Q^H( U_s +\frac{U_s\B w_s \B z_sU_s }{1-\B z_s^H U_s^H \B w_s})=R^{-H},
\]
which gives
\[
U_s=Q R+\B z_s\B w_s^H= Q R^{-H} - \frac{U_s\B w_s \B z_sU_s }{1-\B z_s^H U_s^H \B w_s}.
\]
Since $R^{-H}$ is upper triangular we have that $Q R^{-H}$ has the same upper staircase shape as 
$Q$ and, therefore, from Lemma \ref{0shift} we conclude that 
\[
\rank(U_s(1:2j, 2(j+1)+1:n))\leq 1, \quad 1\leq j\leq \lfloor \frac{n+1}{2}\rfloor-2, \ s\geq 0.
\]
The argument stated in Remark \ref{rem1} extends
this property to  a possibly singular $A_0$ and a fortiori $A_s$, $s\geq 0$. 
\end{proof}

\begin{remark}\label{rem2}
It is worth pointing out that although the  rank  structure of $U_s$ is closed in the limit its 
parametrization via generators is not \cite{note}.  This means that the rank one  representation 
of the entries of $U_s$ located 
 in the upper triangular portion does not hold  in the general case where  the starting matrix $A_0$ 
can be singular. 
\end{remark}

From the previous theorem we derive a  structural representation of each matrix $A_s$, $s\geq 0$, generated 
under the QR process \eqref{qrgeneral} applied to $A_0=\widehat C$  given as in \eqref{ones}.  In the next section 
we provide a  fast adaptation of this process based on the relations \eqref{base0},\eqref{base1}, and \eqref{base2}.

\section{Fast Algorithms and Numerical Results}
\setcounter{equation}{0}

In this section we devise a fast adaptation of the QR iteration \eqref{qrgeneral} applied to a starting  invertible 
matrix $A_0=\widehat C\in \mathbb C^{n\times n}$  given as in \eqref{ones} by using the structural properties described above. 
Let us first observe that each matrix $A_s$, $s\geq 0$, generated by \eqref{qrgeneral} can be represented by means of 
the following sparse data set of size $O(n)$:
\begin{enumerate}
\item the nonzero entries of the banded matrix $\widehat A_s\in \mathbb C^{n\times n}$  obtained from $A_s$  according to 
\[
\widehat A_s=(\hat a_{i,j}^{(s)}), \quad \hat a_{i,j}^{(s)}=
\left\{\begin{array}{ll}
0, \  {\rm if} \ j\geq 2\lfloor \frac{i+1}{2}\rfloor+3, \ 1\leq i\leq 2\lfloor \frac{n+1}{2}\rfloor-4; \\
a_{i,j}^{(s)}, \ {\rm elsewhere}; \end{array}\right.
\]
\item  the vectors $\B z_s=(z_i^{(s)}), \B w_s=(w_i^{(s)})\in \mathbb C^n$ and $\B f_s\colon=U_s \B w_s, \B f_s=(f_i^{(s)})$, 
and  
$\B g_s\colon=U_s^H \B z_s, \B g_s=(g_i^{(s)})$. 
\end{enumerate}
The nonzero pattern of the matrix $\widehat A_s$  looks as below: 
\[
\widehat A_s=\left[\begin{array}{cccccccc}
\star & \star & \star & \star\\
\star & \star & \star &\star  \\
& \star & \star & \star & \star &\star\\
& \star & \star & \star & \star &\star\\
& & &  \star & \star & \star & \star &\star\\
&  & & \star & \star & \star & \star &\star\\
& & &  & &  \star & \star & \star  \\
&  & &  & & \star & \star & \star \\
\end{array}\right], \quad \quad  (n=2k),
\]
or 
\[
\widehat A_s=\left[\begin{array}{ccccccc}
\star & \star & \star& \star\\
\star & \star & \star  & \star\\
& \star & \star & \star & \star& \star\\
& \star & \star & \star & \star &\star\\
& & &  \star & \star & \star & \star\\
&  & & \star & \star & \star & \star\\
& & &  &  & \star & \star \\
\end{array}\right], \quad  \quad (n=2k-1).
\]
From \eqref{base0} and \eqref{base2} we find that the entries of the matrix $A_s=(a_{i,j}^{(s)})$ can be 
expressed in terms of elements of this data set as follows: 
\begin{equation}\label{fastrep}
 a_{i,j}^{(s)}=
\left\{\begin{array}{ll}
-\sigma^{-1} f_i^{(s)}\bar{g_j}^{(s)}- z_i^{(s)}\bar{w_j}^{(s)}, \  {\rm if} \ j\geq 2\lfloor \frac{i+1}{2}\rfloor+3, \ 1\leq i\leq 2\lfloor \frac{n+1}{2}\rfloor-4; \\
\widehat a_{i,j}^{(s)}, \ {\rm elsewhere}; \end{array}\right.
\end{equation}
 where $\sigma=1-\B z_s^H U_s^H \B w_s=1-\B z_0^H U_0^H \B w_0 =1-\B e_3^H \B p=1-1-\displaystyle\frac{\bar p_0}{\bar p_n}=-
\displaystyle\frac{\bar p_0}{\bar p_n}$.
The next procedure performs a  structured variant of the QR iteration \eqref{qrgeneral} applied 
to an initial matrix $A_0=\widehat C\in \mathbb C^{n\times n}$  given as in \eqref{ones}.

\medskip
\framebox{\parbox{8.0cm}{
\begin{code1}
{\bf Procedure} {\bf Fast\_QR }\\
{\bf Input}: $\widehat A_s$, $\sigma$, $\B z_s$, $\B w_s$, $\B f_s$, $\B g_s$;\\
{\bf Output}: $\widehat A_{s+1}$, $\sigma$, $\B z_{s+1}$, $\B w_{s+1}$, $\B f_{s+1}$, $\B g_{s+1}$; \\
\quad  1. Compute the shift $\rho_s$. \\
\quad  2. Find the factored form  \eqref{qrep} of the matrix $Q_s$  such that\\ 
\quad \quad \quad \quad 
$
Q_s^H(A_s-\rho_s I)=R_s, \quad R_s \ {\rm upper} \ {\rm triangular},
$
\\
 \quad  \quad where $A_s$ is represented via \eqref{fastrep}. \\
\quad 3.  Determine  $\widehat A_{s+1}$ from the entries of $A_{s+1}=Q_s^H A_s Q_s$.\\
\quad 4. Evaluate  $\B z_{s+1}=Q_s^H\B z_s$, $\B w_{s+1}=Q_s^H\B w_s$,  $\B f_{s+1}=Q_s^H\B f_s$,  
$\B g_{s+1}=Q_s^H\B g_s$.
\end{code1}}}
\bigskip

The factored form of $Q_s$ makes it possible to execute the steps $2,3$ and $4$ simultaneously by improving the 
efficiency of computation. The matrix $A_s$  is represented by means of four vectors and a 
diagonally structured matrix  $\widehat A_s$  encompassing the band profile of $A_s$.  This matrix  
could be stored in a 
rectangular array  but for the sake of simplicity in our implementation 
we adopt the MatLab\footnote{Matlab is a registered trademark of The Mathworks, Inc..}  sparse matrix format.
Due to the occurrences of deflations the QR process is applied to a principal submatrix of $A_s$ starting at 
position $pst+1$ and ending at position $n-qst$, where $pst=qst=0$ at beginning. 
At the core of ${\bf Fast\_QR }$ there is a  structured  adaptation of the  QR iteration  
applied to 
$B=A_s(pst+1:n-qst,pst+1:n-qs)-\rho_s\, I_{n-pst-qst}$. In particular, we compute the Givens reflector $\mathcal G_1(\bar \gamma_1, \bar \sigma_1)$ of equation~\eqref{qrep} based on the shit $\rho_s$ computed in step 1, and we perform the similarity transformation
$$
B_1=\mathcal G_1(\bar \gamma_1, \bar \sigma_1)^H B\,  \mathcal G_1(\bar \gamma_1, \bar \sigma_1).
$$
This is done by using only the representation of $B$, that is the portion of the four vectors $\B f$, $\B g$, $\B z$, $\B w$ with indices between $pst+1$ and $n-qst$ and $\widehat A_s(pst+1:n-qst,pst+1:n-qs)$,  and acting only on the first two rows  and columns of them. 

Then, defining $ndim=n-pst-qst$ the dimension of $B$ and for $\ell=2:2: 2*(\lfloor(ndim+1)/2\rfloor -2)$ we compute the matrices $ \mathcal G_{\ell,\ell+1}$ as the unitary factor of a $QR$ factorization of the $3\times 3$ diagonal blocks $B_1(\ell:\ell+2, \ell:\ell+2)$, and updating $B_1$ as follows
$$
B_1=\mathcal G_{\ell,\ell+1}^H B_1\,  \mathcal G_{\ell,\ell+1}.
$$

As in equation~\eqref{qrep}, the last unitary transformation $\mathcal G_{n-1}$ is computed in a different way  in the odd and in the even case.

Despite the simplicity of this scheme we have to  deal carefully with the representation of $B$ in order to update the banded matrix and the four generators.

The computation of the shift $\rho_s$  at the first step of ${\bf Fast\_QR }$ can be  carried out by 
several strategies \cite{MC}. In our implementation we employ the Wilkinson idea  by choosing as a shift one of
the roots of the trailing 2-by-2 submatrix of $A_s(pst+1:n-qst,pst+1:n-qs)$
(the one closest to the final entry). For an input companion matrix  expressed as a rank--one correction 
of a unitary  CMV--like shaped matrix   this technique ensures zero shifting at the early iterations. It has been 
observed experimentally that this fact is important for the correct fill in  both in the rank--two structure in the 
 upper triangular part  and in the band profile of $A_s$.
Incorporating the Wilkinson shifting within the explicit shifted $QR$ 
method  ${\bf Fast\_QR}$ and implementing a step of $QR$ iteration on the representation as just described, yields our proposed  fast CMV--based eigensolver for companion matrices. The algorithm has 
been implemented in MatLab and tested on several examples. 
This implementation can be obtained from the authors upon request.

In order to check the accuracy of the  output 
we compare the computed approximations 
 with the ones returned by the internal function {\tt eig}  applied to the initial companion matrix 
$C=C(p)\in \mathbb C^{n\times n}$ 
without the balance option. Specifically, we match the two lists  of approximations and then  find the 
average absolute error $err=\sum_{j=1}^n err_j/n$.

For a backward stable algorithm in the light of the classical perturbation results for 
eigenvalue computation \cite{MC} we know that  this error  would be of the order of
$\|\Delta C \|_\infty \,\mathcal K_\infty(V) \, {\tt \varepsilon}$, where $\|\Delta C \|_\infty$ is the backward error, 
$\mathcal K_\infty(V) =\| V \|_\infty  \cdot \|V^{-1}\|_\infty$ is the 
condition number of $V$,  the eigenvector matrix  of $C$
and ${\tt \varepsilon}$ denotes the machine precision.  A backward stability analysis  of the customary 
QR eigenvalue algorithm is performed in  \cite{Tisse} by showing that $\|\Delta C \|_F\leq c n^3 \| C \|_F$ 
for a small integer constant $c$.  A partial extension of this result to certain 
fast adaptations of the QR algorithm for rank--structured matrices is provided in \cite{EGGnew} 
by replacing $\| C \|_F$ with a measure of the magnitude of the generators. The 
numerical experience reported in \cite{last}
further support this extension.   In the present case we find that 
\begin{eqnarray*} 
\|C \|_\infty=\|A_0\|_\infty&\le& \|\widehat A_0\|_\infty+\|\sigma^{-1} \B f_0\|_\infty\| \B g_0\|_\infty +\| \B w_0\|_\infty\| \B z_0\|_\infty\\ \label{magg}
&=& \|\widehat A_0\|_\infty+\|\sigma^{-1} \B f_0\|_\infty+\| \B w_0\|_\infty.
\end{eqnarray*}
The parameter  $\sigma^{-1}=-{\bar p_n}/{\bar p_0}$  in the starting  representation  
via generators should be incorporated into the vector $\B f_0$, leading 
 to a vector whose entries depend on the ratios 
$\pm p_j/p_0$. Viceversa, the entries of vector $\B w_0$, depend on the ratios $\pm p_j/p_n$. 
When the coefficients of the polynomial $p(z)$ are unbalanced, to keep trace of the possible 
unbalanced entries of both $\B f_0$ or $\B w_0$, we may consider  the maximum expected error as 
$nne=\left( \|\widehat A_0\|_\infty+\|\sigma^{-1} \B f_0\|_\infty+\| \B w_0\|_\infty\right)\, \mathcal K_\infty(V) \,{\tt \varepsilon}$. Our implementation reports as output the value of 
$werr=err/nne$. In accordance with our claim this quantity  would be bounded by a small multiple of 
$n^3$.

As a measure of efficiency of the algorithm  we also determine 
the  average number  $averit$ of QR steps per eigenvalue.

We have performed many numerical experiments with real polynomials of both small and large degree.  Moreover, to 
support our expectation about  roundoff errors  we consider several cases where the  input polynomial 
is (anti)palindromic  in such a  way that $\|\sigma^{-1}\B f_0\|_\infty= \| \B w_0\|_\infty $. Our test suite consists of the following polynomials: 
\begin{itemize}
\item (P1) $p(z)=1+ (\frac{n}{n+1}+\frac{n+1}{n})z^n +z^{2n}$ \cite{BFGM}. The zeros can be explicitly determined and 
 lie on two circles centered at the origin that are poorly separated.
\item (P2) $p(z)=\frac{1}{n} \left(\sum_{j=0}^{n-1} (n+j)z^j  +(n+1)z^n+\sum_{j=0}^{n-1} (n+j)z^{2n-j}\right)$ \cite{BH}. 
This is another test problem for spectral factorization algorithms. 
\item (P3) $p(z)=(1-\lambda)z^{n+1} -(\lambda+1)z^n +(\lambda+1)z -(1-\lambda)$ \cite{AGDP}.  
This family of antipalindromic 
polynomials arises  in the context of 
a boundary--value problem  whose eigenvalues coincides with the zeros of an entire function related with $p(z)$.
\item (P4) A collection of small--degree polynomials \cite{TT}:
\begin{enumerate}
\item the Bernoulli polynomial 
$p(z)=\sum_{j=0}^n\left(\begin{array}{cc}n\\j\end{array}\right)b_{n-j}z^j$, where $b_j$ are the Bernoulli numbers;
\item the Chebyshev polynomial of first kind; 
\item  the partial sum of the exponential $p(z)=\sum_{j=0}^n  (2z)^j/{j!}$.
\end{enumerate}
\item(P5) Polynomials $p(z)=\sum_{j=0}^np_jz^j$ with coefficients of the form $p_j=a_{j}\times 10^{e_{j}}$, 
where $a_{j}$ and $e_{j}$ are drawn from the uniform distribution in $[-1,1]$ and $[-3,3]$, 
respectively. These polynomials were proposed in \cite{JT} for testing purposes. 
\item (P6) The symmetrized version of the previous polynomials, that is, $p(z)=s(z)s(z^{-1})z^n$  where 
$s(z)=\sum_{j=0}^ns_jz^j$ with coefficients of the form $s_j=a_{j}\times 10^{e_{j}}$ and  $a_{j}\in [-1,1]$ and 
$e_j\in [-3,3]$.
\end{itemize}

Table \ref{table1} shows the numerical 
results for the first three sets of symmetric polynomials.  For the sake of illustration in 
Figure \ref{f1} and \ref{f2} we also  display the distribution of the zeros computed by our 
routine and the MatLab function {\tt eig} applied to polynomials in the class $P2$ and $P3$, respectively.

\begin{table}
\centering
\begin{tabular}{ccccccccc}
\toprule
Test Set Number&  $n$    & $nne/{\tt \varepsilon}$ & $err$ & $werr$ &  $averit$   \\ 
\midrule
\multirow{5}*{$P1$}  
& 64 &  4.14e+04& 4.12e-14 & 7.47e-03 &  4.55  \\ 
& 128 &  1.65e+05& 1.16e-13 & 5.29e-03 &  4.53  \\ 
& 256 &  6.57e+05 &2.84e-13 & 3.24e-03 &  4.51  \\ 
& 512 &  2.62e+06 &8.87e-13 & 2.54e-03 &  4.51  \\
& 1024 & 1.05e+07 &2.61e-12 & 1.87e-03 & 4.51 \\ 
\midrule
\multirow{5}*{$P2$}
& 64 &  2.36e+05 &3.94e-12 &  7.52e-02 & 3.66  \\ 
& 128 & 1.62e+06 &1.03e-10 &  2.86e-01 & 3.44  \\ 
& 256 & 1.13e+07 &1.21e-09 &  4.84e-01 & 3.23  \\ 
& 512 & 8.01e+07 &2.73e-08 &   1.53e+00 & 3.06  \\ 
& 1024 &  5.77e+08& 5.45e-06 & 4.26e+01 & 2.97  \\
\midrule
\multirow{5}*{$P3 (\lambda=0.9)$}
& 64 &  1.10e+04 &4.12e-15 &  1.69e-03 & 2.94  \\ 
& 128 & 2.20e+04 &1.07e-14 & 2.18e-03 & 2.67  \\ 
& 256 &4.41e+04 &2.83e-14 &  2.88e-03 & 2.57  \\ 
& 512 & 8.83e+04 &3.83e-14 & 1.96e-03 & 2.53  \\ 
& 1024  & 1.77e+05& 4.19e-14 & 1.07e-03 & 2.51 \\ 
\midrule
\multirow{5}*{$P3 (\lambda=0.999)$}
& 64  & 1.08e+06 &6.48e-15 &  2.71e-05 & 3.03  \\ 
& 128  & 2.16e+06& 9.97e-15 & 2.08e-05 & 2.71\\ 
& 256  & 4.34e+06 &2.50e-14 & 2.59e-05 & 2.58\\ 
& 512  & 8.68e+06 &3.66e-14 & 1.90e-05 & 2.54 \\ 
& 1024 & 1.74e+07 &4.25e-14 & 1.10e-05 & 2.52  \\ 
\bottomrule
\end{tabular}
\caption{\label{table1} Numerical results for the sets $P1$, $P2$ and $P3$ of (anti)palindromic polynomials}
\end{table}

\begin{figure}
\centering
\includegraphics[scale=0.4]{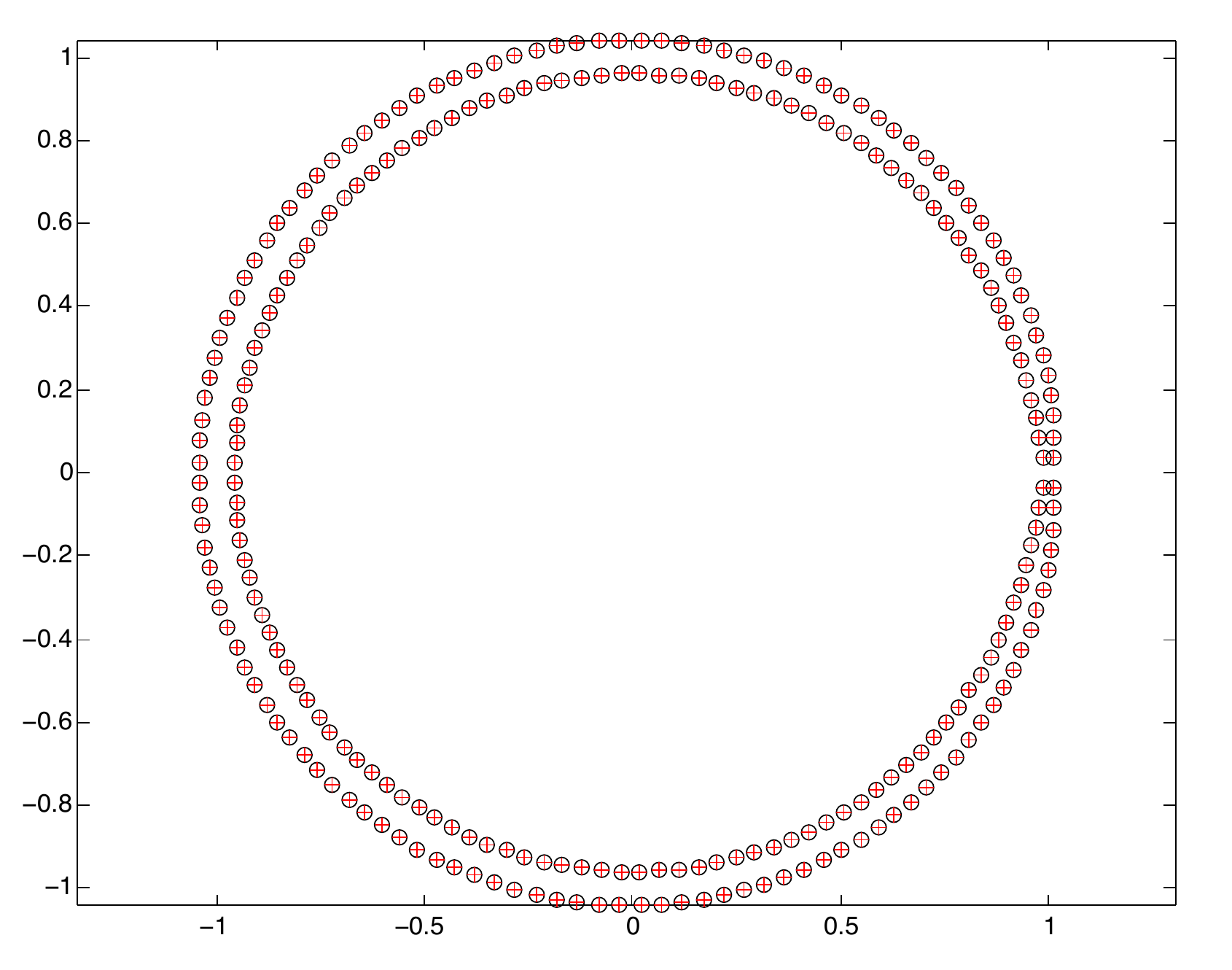}
\caption{Distribution of the zeros computed by our routine (red plus) and {\tt eig} (black circles) for 
the polynomial in the class  $P2$ of degree $n=128$.}
\label{f1}
\end{figure}

\begin{figure}
\centering
    \subfigure[]{\includegraphics[scale=0.3]{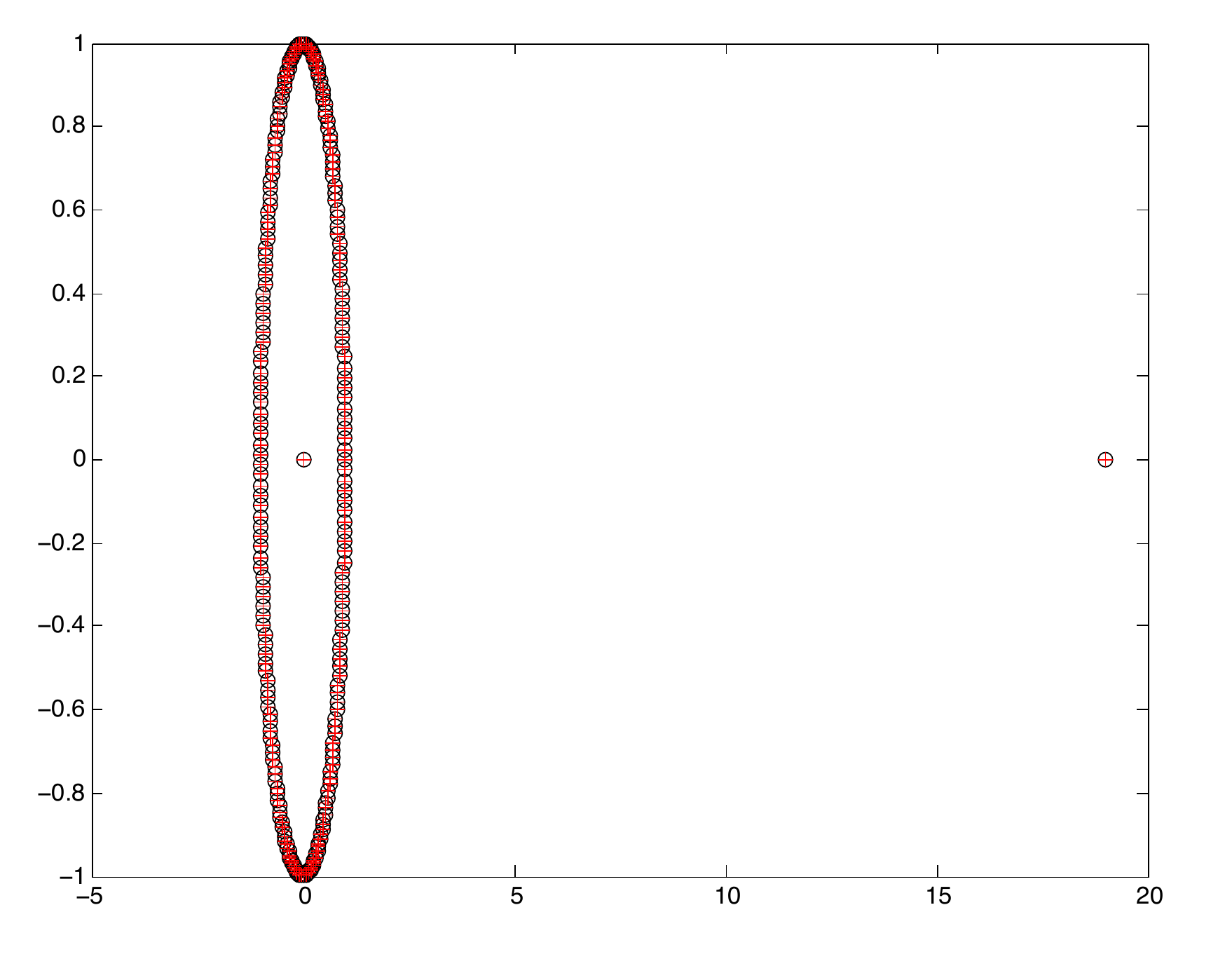}}\hspace{2cm}
    \subfigure[]{\includegraphics[scale=0.3]{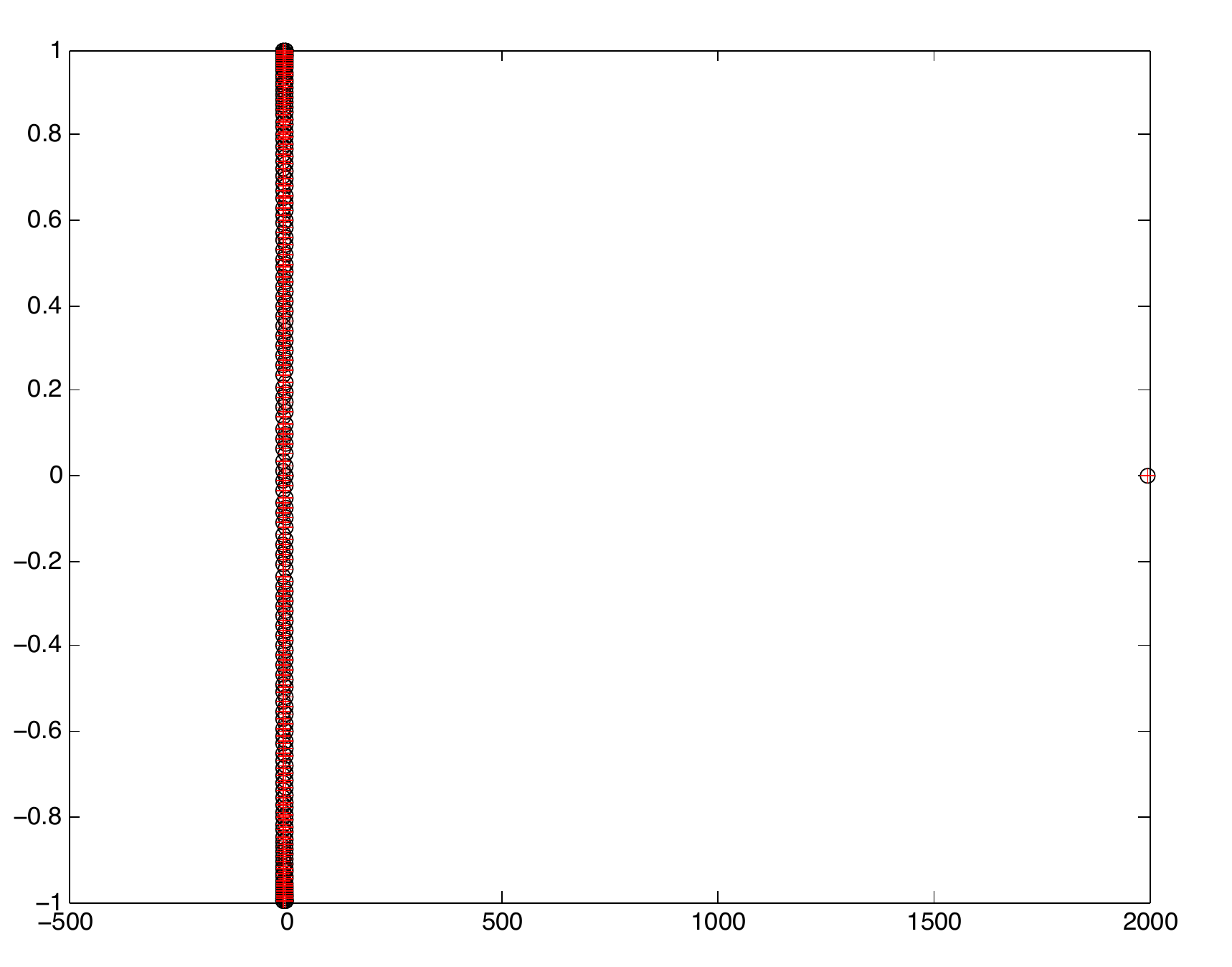}} 

  \caption{Distribution of the zeros computed by our routine (red plus)  and {\tt eig} (black circles) for 
the polynomials in  the class $P3$ of degree $n=128$ with $\lambda\in \{0.9,0.999\}$. \label{f2}}
 \end{figure}

A certain degeneration of the accuracy of computed results can be observed in example $P2$ in Table \ref{table1}, 
but this is within the bounds provided by the backward error analysis. 

Table \ref{table2} shows the numerical 
results for the small degree polynomials $P4$.  For the sake of illustration in 
Figure \ref{f3} and \ref{f4} we also  display the distribution of the zeros computed by our 
routine and the MatLab function {\tt eig} applied to polynomials in the class $P4(1-2)$ and $P4(3)$, respectively.

\begin{table}
\centering
\begin{tabular}{ccccccccc}
\toprule
Test Set Number&  $n$ &   $nne/{\tt \varepsilon}$ & $err$  & $werr$&
 $averit$   \\ 
\midrule
\multirow{5}*{$P4(1)$}  
& 10 &  5.01e+05 &2.75e-14 & 2.47e-04 & 3.50 \\ 
& 20 &  1.34e+13 &2.47e-13 & 8.31e-11 & 3.50 \\ 
& 30 &  5.94e+25 &2.02e-12 & 1.53e-22 & 3.77 \\ 
\midrule
\multirow{5}*{$P4(2)$}
& 10 &  4.82e+06 &5.34e-12 &  4.99e-03 & 3.40  \\ 
& 20 &  1.69e+14 &3.52e-05 &  9.41e-04 & 3.40  \\ 
& 30 &  6.27e+21 &1.89e-01 &  1.36e-07 & 4.03  \\ 
\midrule
\multirow{5}*{$P4(3)$}
& 10&  2.93e+08 &3.12e-14 & 4.79e-07 & 3.20 \\ 
& 20 & 9.83e+25 &1.27e-11 & 5.81e-22 & 3.35 \\ 
& 30 & 8.49e+47 &3.77e-08 & 2.00e-40 & 3.30\\ 
\bottomrule
\end{tabular}
\caption{\label{table2} Numerical results for the sets $P4(1-3)$.}
\end{table}

\begin{figure}
\centering
    \subfigure[]{\includegraphics[scale=0.369]{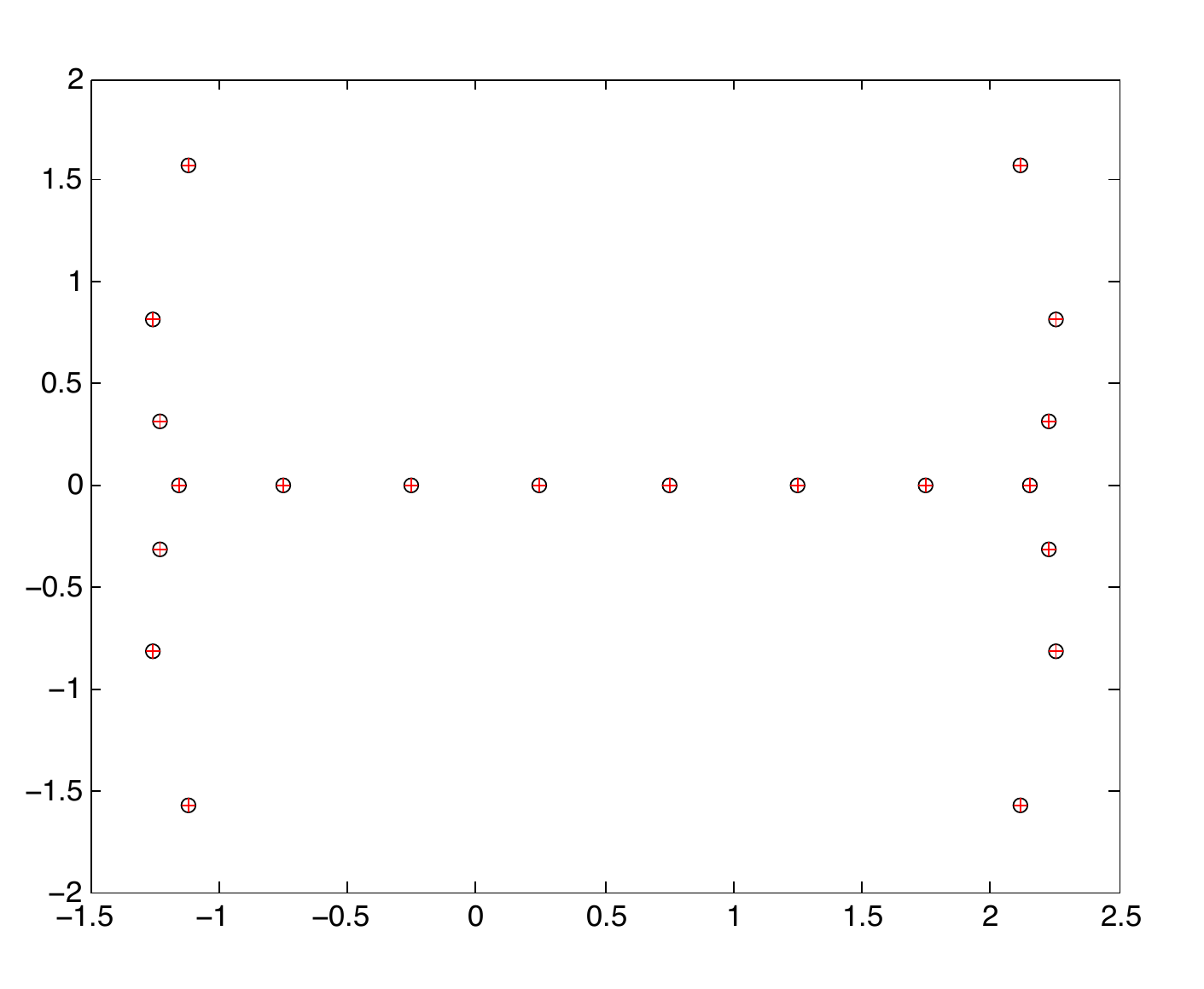}}\hspace{2cm}
    \subfigure[]{\includegraphics[scale=0.3]{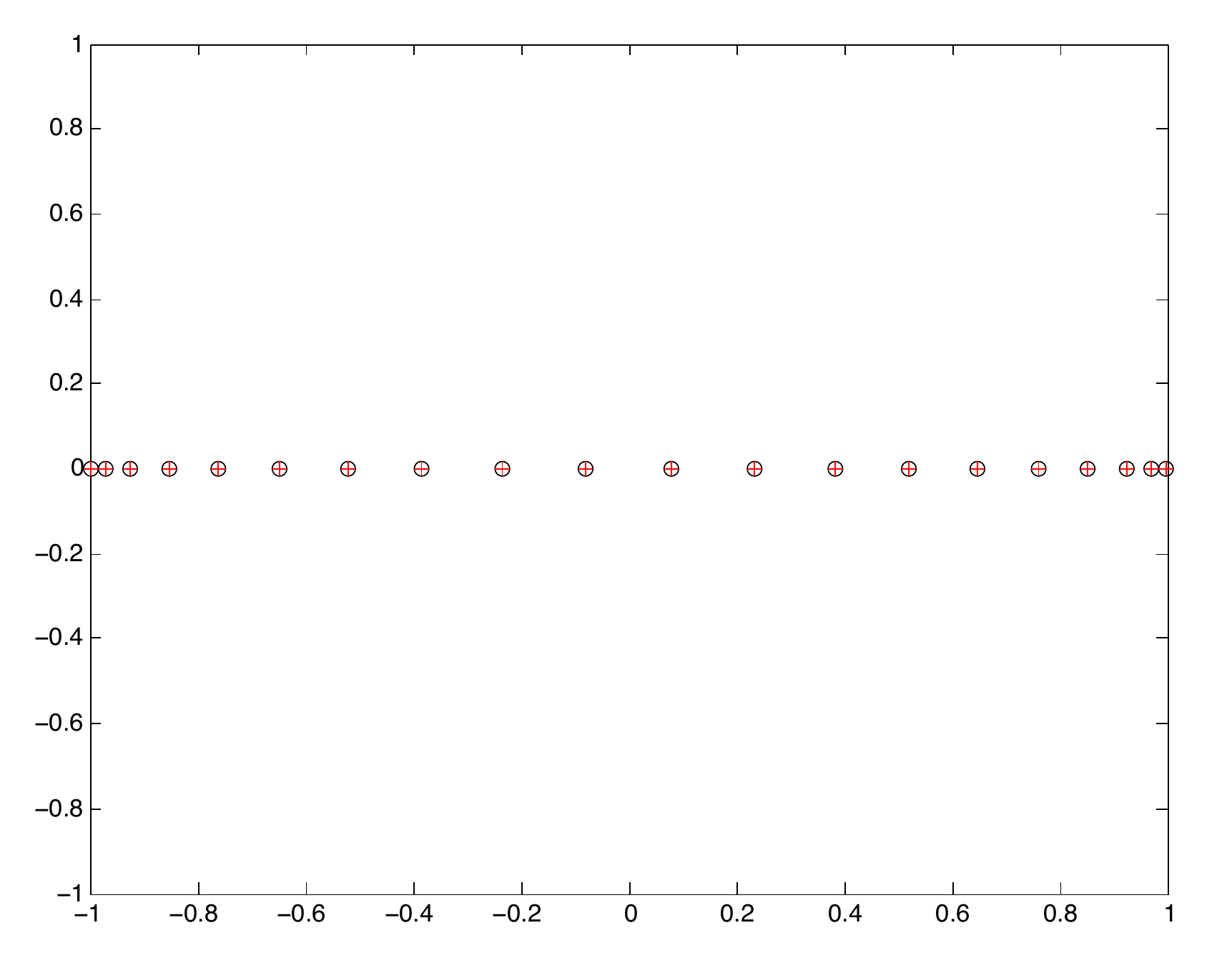}} 
  \caption{Distribution of the zeros of Bernoulli and Chebyshev polynomial of degree $20$ 
computed by our routine (green diamonds) and {\tt eig} (red circles).  \label{f3}}

\end{figure}
 
It is worth pointing out the loss of information in the Chebyshev case due to the usage of 
generators depending on the normalization for both the leading and the trailing coefficient of the 
polynomial.  This is a  potential drawback of our approach.  

\begin{figure}
\centering
    \subfigure[]{\includegraphics[scale=0.3]{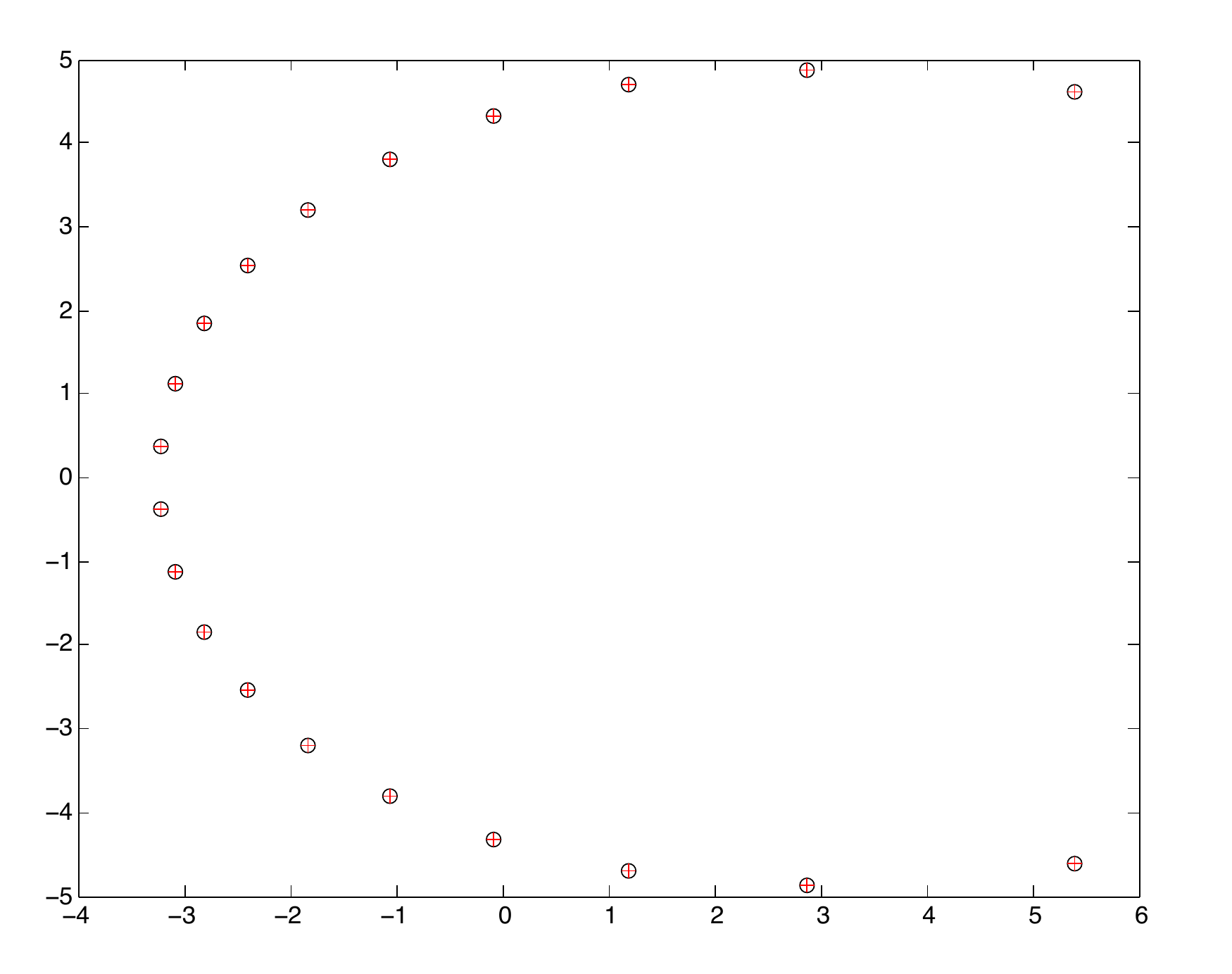}}\hspace{2cm}
    \subfigure[]{\includegraphics[scale=0.3]{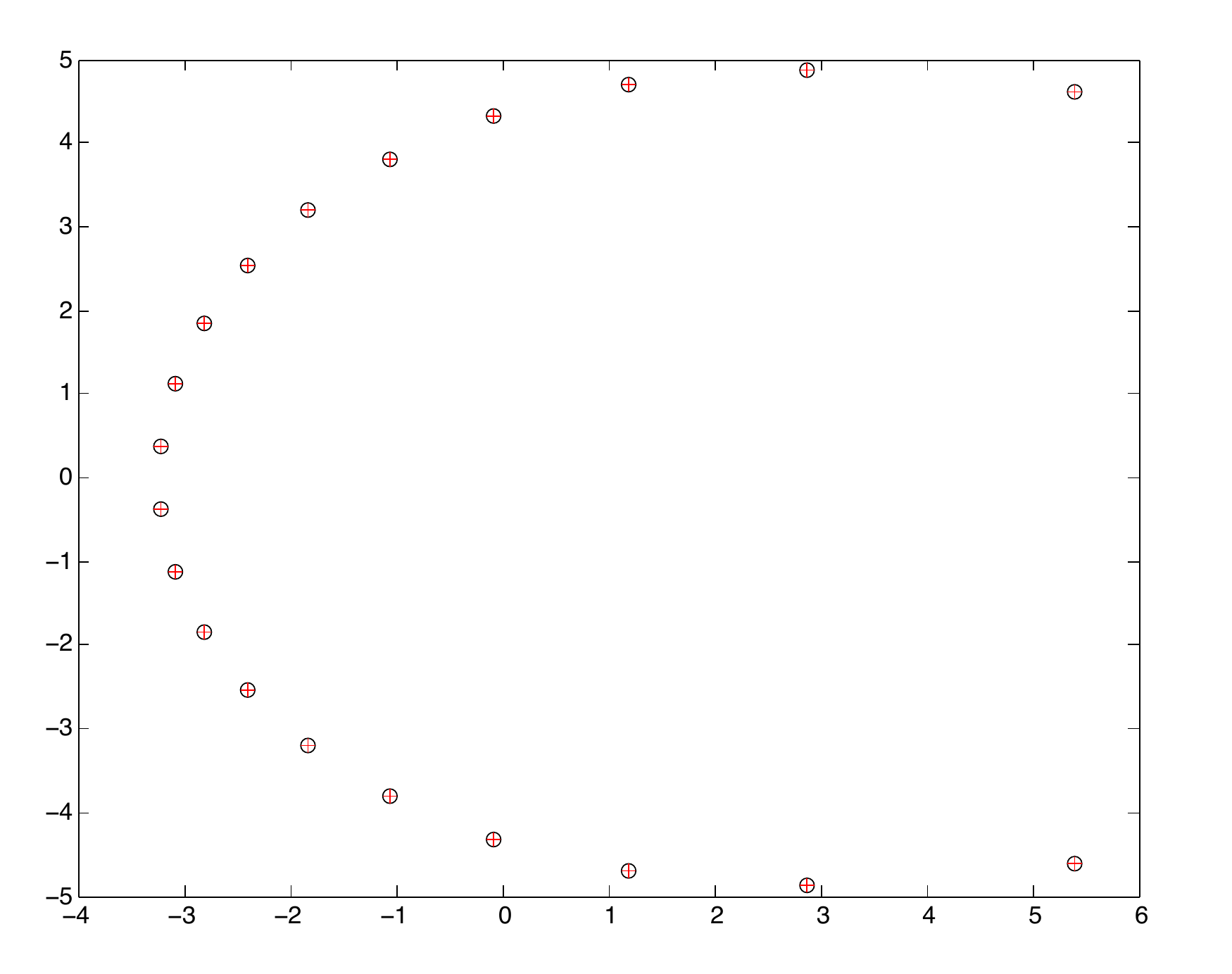}} 
  \caption{Distribution of the zeros of truncated Taylor series of $e^{\displaystyle{2 z}}$  of degree $20$ 
 and $30$ computed by our routine (green diamonds) and {\tt eig} (red circles).  \label{f4}}

\end{figure}

Table \ref{table3} finally gives the  numerical 
results for the  polynomials $P5$ and $P6$. Here we report for $nne/{\tt \varepsilon}$ the min/max range 
 and for the other columns the maximum value 
 of the data output 
variables over fifty experiments. 

\begin{table}
\centering
\begin{tabular}{ccccccccc}
\toprule
Test Set Number&  $n$ &  $nne/{\tt \varepsilon}$ & $err$ &  $werr$&
 $averit$   \\ 
\midrule
\multirow{5}*{$P5$}  
& 32 &  4.76e+05 - 1.49e+20 & 7.50e-03&  1.94e-01 & 3.67 \\ 
& 64 &  2.87e+03 - 3.66e+19 & 5.33e-04&  2.40e-03 & 3.65 \\ 
& 128 &  9.90e+04 - 7.43e+19 & 4.48e-03&  1.47e-01 & 3.41 \\
\midrule
\multirow{5}*{$P6$}
& 16 &  2.55e+03 - 5.47e+19 & 1.71e-02 & 6.16e-03 & 3.53\\
 & 32 & 7.64e+04 - 2.49e+22 & 1.34e-02 & 9.48e-03 & 3.61\\ 
& 64 &  1.08e+06 - 2.13e+20 & 4.76e-02  & 1.51e-02 & 3.42\\ 
& 128 &  1.46e+07 - 6.96e+23 & 1.40e-01 &  8.71e+00 & 3.33\\ 
\bottomrule
\end{tabular}
\caption{\label{table3} Numerical results for the sets $P5$, $P6$.}
\end{table}

\section{Conclusion and Future Work}
\setcounter{equation}{0}
In this paper we have presented a novel fast  QR--based  eigensolver for companion matrices 
exploiting the structured technology for CMV--like representations.  To our knowledge this is the first 
numerically reliable fast adaptation of the QR algorithm for perturbed unitary matrices which 
makes use of only four  vectors to express the rank structure of the matrices generated
 under the iterative process. As a result, we obtain a data sparse parametrization of these matrices 
which 
at the same time is able to capture the
structural properties of the matrices and yet  to be sufficiently easy to manipulate and update
 for computations. Although very promising, some numerical issues
 associated with the proposed approach are still under investigation. The first one is a certain 
sensibility of the algorithm in the initial  steps where the  band profile of the matrix is filled 
using the information propagated  from the polynomial coefficients.   The second issue is 
concerned with the magnitude of the generator vectors 
depending on the normalization for both the leading and the trailing coefficient of the 
polynomial. Both these problems can be circumvented by using different representations of the 
rank--two structure. 
Finding the right balance between robustness and efficiency is the main subject of our current research.

\bibliographystyle{plain}

\end{document}